\documentclass{amsart}

\usepackage{amssymb}

\usepackage{graphicx,caption,subcaption}

\newtheorem{theorem}{Theorem}
\newtheorem{lemma}{Lemma}

\theoremstyle{definition}

\theoremstyle{remark}
\newtheorem{remark}{Remark}

\numberwithin{equation}{section}

\allowdisplaybreaks[4]

\begin{document}
\nocite*
\title{Variance of Lattice Point Counting in Thin Annuli}


 \author[L. Colzani]{Leonardo Colzani}
\address{Dipartimento di Matematica e Applicazioni,
   Universit\`a degli Studi di Milano-Bicocca,
   Via R. Cozzi 55, 20125 Milano,
   Italy}
\curraddr{}
\email{leonardo.colzani@unimib.it}
\thanks{}

\author[B. Gariboldi]{Bianca Gariboldi}
\address{Dipartimento di Ingegneria Gestionale, dell'Informazione e della Produzione,
  Universit\`a degli Studi di Bergamo,
  Viale Marconi 5, 24044 Dalmine (BG),
  Italy}
\curraddr{}
\email{biancamaria.gariboldi@unibg.it}
\thanks{}

\author[G. Gigante]{Giacomo Gigante}
\address{Dipartimento di Ingegneria Gestionale, dell'Informazione e della Produzione,
  Universit\`a degli Studi di Bergamo,
  Viale Marconi 5, 24044 Dalmine (BG),
  Italy}
\curraddr{}
\email{giacomo.gigante@unibg.it}
\thanks{}

\date{}

\dedicatory{}

\begin{abstract}
We give asymptotic estimates of the variance of the
number of integer points in translated thin annuli in any dimension.
\end{abstract}

\maketitle

Y. G. Sinai proved in \cite{Si} that the number of integer points in the plane
inside a thin annulus of fixed area $\lambda$, of random shape and large random radius, with a
suitable definition of randomness, converges in distribution to a Poisson
random variable with parameter $\lambda $. The probabilistic proof does not
exhibit a specific annulus.
See also \cite{M, Mi}.
 Indeed in \cite{M} it is shown that the number of integer points in the circular annulus $\{r-1/4r < |x| \leq r+1/4r \}$ in the plane does not converge to a Poisson distribution when $ r  $ varies randomly and uniformly in $\left[a_1 L,a_2 L\right]$ and $L$ goes to $+\infty$. 
The reason is that, under the condition that the annulus contains some integer points, then with probability almost one the number of integer points in the annulus tends to infinity.
On the other hand, a translation of the annulus breaks the symmetry, and the situation changes.
Indeed Z. Cheng, J. L. Lebowitz, P. Major proved in
\cite{CLM} that if $\Omega $ is a convex set in the plane with a
smooth boundary with positive curvature, then the expectation and variance
for the number of integer points in a shifted annular region of radius $r$ and thickness $c/r$ 
\[
\left[\left(
r+c/\left( 2r\right) \right) \Omega-x\right] \setminus\left[\left( r-c/\left( 2r\right) \right)
\Omega-x\right],
\] 
where $x$ is uniformly distributed in the unit square,
 are both asymptotic to the area of the annulus $2c\left\vert 
\Omega \right\vert $ as $c$ is fixed and $r\rightarrow +\infty $. Since the mean and the
variance of a Poisson distribution coincide, this is
consistent with the conjecture that this random variable converges in
distribution to a Poisson random variable. Indeed these authors briefly
mention higher dimensional analogues. The following is a proof of these
higher dimensional analogues via Fourier analysis.

\begin{theorem} \label{T1}
Assume that $\Omega $ is a convex body
in $\mathbb{R}^{d}$ with smooth boundary with everywhere positive
Gaussian curvature, which contains in its interior the origin. Denote by 
$\Omega \left( r,t\right) $ the annulus $\left( r+\left( t/2\right)
\right) \Omega \setminus \left( r-\left( t/2\right) \right) \Omega $, and by 
$\left\vert \Omega \left( r,t\right) \right\vert $ its volume. Then
for every $\alpha >\left( d-1\right) /\left( d+1\right) $ there
exists $0<\beta <1$ and a positive constant $C$ such
that for every $1\leq r<+\infty $ and every $0<t\leq r^{-\alpha }$one has 
\[
\left\vert \int_{\mathbb{T}^{d}}\left\vert \sum_{k\in \mathbb{Z}^{d}}\chi
_{\Omega \left( r,t\right) - x}\left( k\right) -\left\vert \Omega \left(
r,t\right) \right\vert \right\vert ^{2}dx-\left\vert \Omega \left(
r,t\right) \right\vert \right\vert \leq C\left\vert \Omega \left( r,t\right)
\right\vert t^{\beta }.
\]
\end{theorem}

The mean of the random variable that counts the number of integer points in
the annulus is the volume of the annulus, hence the above is an estimate of
the variance of this random variable. In particular, the theorem can be
rephrased by saying that the expectation and variance of the number of
integer points in translated annuli are asymptotic when  $ r \rightarrow + \infty $ and $t\rightarrow 0+$,
with $t\leq r^{-\alpha }$ for some $\alpha >\left( d-1\right) /\left(
d+1\right) $. Observe that when $t = o(r)$, then 
\begin{equation*}
\left\vert \Omega \left( r,t\right) \right\vert =\left\vert \left(
r+t/2\right) \Omega \right\vert -\left\vert \left( r-t/2\right) \Omega
\right\vert =\left( \left( r+t/2\right) ^{d}-\left( r-t/2\right) ^{d}\right)
\left\vert \Omega \right\vert \sim dr^{d-1}t\left\vert \Omega \right\vert .
\end{equation*}

In particular, under the assumption that $0<t\leq r^{-\alpha }$ with $\alpha
>\left( d-1\right) /\left( d+1\right) $, the measure of the annulus $%
\left\vert \Omega \left( r,t\right) \right\vert \sim dr^{d-1}t\left\vert
\Omega \right\vert $ may diverge. Also observe that with the above theorem
in dimension $d=2$ and with $r=c/t$ one recovers the results in
\cite{CLM}, and indeed the assumption $t=c/r$ can be replaced by
the weaker assumption $t\leq r^{-\alpha }$ for some $\alpha >1/3$. We do not
know if this assumption $0<t\leq r^{-\alpha }$ with $\alpha >\left(
d-1\right) /\left( d+1\right) $ can be weakened, but it follows from some
results in \cite{PS} that the only assumption that the widths of
the annuli converge to zero does not suffice, and one has to require a
suitable speed. Finally, also the curvature assumption is necessary. The
variance of annuli with boundary points of zero curvature may be much larger
than the mean, and an asymptotic estimate of the variance may fail. An
example are the annuli generated by polyhedra with faces with rational
orientation. See Remark \ref{R2} below.

The main tool in our proof is the Fourier expansion of the random variable that counts the integer points. As shown by D. Kendall in \cite{K}, an estimate from above of the variance
of the number of integer points in shifted ovals follows easily from
estimates of the order of decay of the Fourier transform of an oval. Here,
in order to obtain an asymptotic for the variance, we shall need to extract
from the Fourier transform more precise geometric informations. We split
this proof in a number of lemmas.

\begin{lemma}\label{L1}
If $\Omega $ is a domain in $\mathbb{R}^{d}$ with finite measure, then the number of integer points in 
$\Omega -x$ is a periodic function of the translation $x$,
and it has the Fourier expansion 
\[
\sum_{k\in \mathbb{Z}^{d}}\chi _{\Omega -x}(k)=\sum_{n\in \mathbb{Z}^{d}}%
\widehat{\chi }_{\Omega }\left( n\right) \exp \left( 2\pi inx\right) .
\]
In particular, this Fourier expansion converges in the square
metric, and 
\[
\int_{\mathbb{T}^{d}}\left\vert \sum_{k\in \mathbb{Z}^{d}}\chi _{\Omega
-x}(k)-\left\vert \Omega \right\vert \right\vert ^{2}dx=\sum_{n\in \mathbb{Z%
}^{d} \setminus \left\{ 0\right\} }\left\vert \widehat{\chi }_{\Omega }\left( n\right)
\right\vert ^{2}.
\]
\end{lemma}

\begin{proof} The first equality is the Poisson summation formula. If one
identifies the torus $\mathbb{T}^{d}=\mathbb{R}^{d}/\mathbb{Z}^{d}$ with the
unit cube $\left\{ -1/2\leq x_{j}<1/2\right\} $, then 
\begin{align*}
\sum_{k\in \mathbb{Z}^{d}}\chi _{\Omega -x}(k) &=\sum_{n\in \mathbb{Z}%
^{d}}\left( \int_{\mathbb{T}^{d}}\sum_{k\in \mathbb{Z}^{d}}\chi _{\Omega
-y}(k)\exp \left( -2\pi iny\right) dy\right) \exp \left( 2\pi inx\right) \\
&=\sum_{n\in \mathbb{Z}^{d}}\left( \int_{\mathbb{R}^{d}}\chi _{\Omega
}(y)\exp \left( -2\pi iny\right) dy\right) \exp \left( 2\pi inx\right)\\
&=\sum_{n\in \mathbb{Z}^{d}}\widehat{\chi }_{\Omega }\left( n\right) \exp
\left( 2\pi inx\right) .
\end{align*}

The second equality is Parseval's identity, just observe that $\widehat{%
\chi }_{\Omega }\left( 0\right) =\left\vert \Omega \right\vert $.
\end{proof}

We emphasize that the above lemma does not claim that the Fourier expansions
of the random variables converge pointwise. Indeed it can be shown that in
dimensions $d=1$ and $d=2$ and for domains with smooth boundaries the above
Fourier expansions are pointwise spherically convergent, but this is not the
case if $d\geq 3$. Anyhow, the series are summable pointwise with suitably
strong summability methods at every point $x$ with $\mathbb{Z}^{d}\cap
\partial \left\{ \Omega -x\right\} =\varnothing $. 

The above lemma suggests
the search of precise estimates of the Fourier transform of an annulus. In
order to guess the correct result, it may be helpful to have an explicit
example. The Fourier transform of the sphere $\left\{ \left\vert
x\right\vert \leq r\right\} $ is a Bessel function, 
\[
\widehat{\chi }_{\left\{ \left\vert x\right\vert \leq r\right\} }\left( \xi
\right) =r^{d}\widehat{\chi }_{\left\{ \left\vert x\right\vert \leq
1\right\} }\left( r\xi \right) =r^{d}\left\vert r\xi \right\vert
^{-d/2}J_{d/2}\left( 2\pi r\left\vert \xi \right\vert \right) .
\]

See \cite[Theorem 4.15, Chapter IV]{SW}. Hence, the Fourier
transform of the annulus $\left\{ r-t/2 < \left\vert x\right\vert \leq
r+t/2\right\} $ is 
\begin{align*}
\widehat{\chi }_{\left\{ r-t/2 < \left\vert x\right\vert \leq
r+t/2\right\} }\left( \xi \right) &=\widehat{\chi }_{\left\{ \left\vert
x\right\vert \leq r+t/2\right\} }\left( \xi \right) -\widehat{\chi }%
_{\left\{ \left\vert x\right\vert \leq r-t/2\right\} }\left( \xi \right) \\
&=r^{d/2}\left\vert \xi \right\vert ^{-d/2}\left( J_{d/2}\left( 2\pi \left(
r+t/2\right) \left\vert \xi \right\vert \right) -J_{d/2}\left( 2\pi \left(
r-t/2\right) \left\vert \xi \right\vert \right) \right) \\
& \ +\left( \left( r+t/2\right) ^{d/2}-r^{d/2}\right) \left\vert \xi \right\vert
^{-d/2}J_{d/2}\left( 2\pi \left( r+t/2\right) \left\vert \xi \right\vert
\right) \\
& \ -\left( \left( r-t/2\right) ^{d/2}-r^{d/2}\right) \left\vert \xi \right\vert
^{-d/2}J_{d/2}\left( 2\pi \left( r-t/2\right) \left\vert \xi \right\vert
\right) .
\end{align*}

Recall the asymptotic expansions of Bessel functions, 
\begin{align*}
J_{v}\left( z\right) &=2^{1/2}\pi ^{-1/2}z^{-1/2}\cos \left( z-\pi \left(
2\nu +1\right) /4\right) +\mathcal{O}\left( z^{-3/2}\right) ,\\
\dfrac{d}{dz}J_{\nu }\left( z\right) &=2^{-1}\left( J_{\nu -1}\left( z\right)
-J_{\nu +1}\left( z\right) \right) \\
&=-2^{1/2}\pi ^{-1/2}z^{-1/2}\sin \left( z-\pi \left( 2\nu +1\right)
/4\right) +\mathcal{O}\left( z^{-3/2}\right) .
\end{align*}

Then, from these formulas and with some trigonometry, one obtains the
asymptotic expansion of the Fourier transform of a spherical shell, 
\begin{align*}
&\widehat{\chi }_{\left\{ r-t/2 < \left\vert x\right\vert \leq
r+t/2\right\} }\left( \xi \right)\\
&=2\pi ^{-1}r^{\left( d-1\right) /2}\left\vert \xi \right\vert ^{-\left(
d+1\right) /2}\cos \left( 2\pi r\left\vert \xi \right\vert -\pi \left(
d-1\right) /4\right) \sin \left( \pi t\left\vert \xi \right\vert \right) \\
& \ +\mathcal{O}\left( r^{\left( d-3\right) /2}t\left\vert \xi \right\vert
^{-\left( d+1\right) /2}\right) .
\end{align*}

When the dimension of the space is odd, the Bessel functions can be written
explicitly in terms of trigonometric functions, and one can also obtain an
exact formula for this Fourier transform in terms of elementary functions.
The behavior of the Fourier transforms of convex bodies and annuli is
similar, although a bit more complicated.

\begin{lemma}\label{L2} 
The Fourier transform of a characteristic
function of a convex body $\Omega $ in $\mathbb{R}^{d}$
with smooth boundary with everywhere positive Gaussian curvature for 
$\left\vert \xi \right\vert \rightarrow +\infty $ has the asymptotic
expansion 
\[
\widehat{\chi }_{\Omega }\left( \xi \right) =a\left( \xi \right) \left\vert
\xi \right\vert ^{-\left( d+1\right) /2}+E\left( \xi \right) .
\]

If $\sigma \left( \pm \xi \right) $ are the points of the
boundary of $\Omega $ with outward unit normals $\pm \xi
/\left\vert \xi \right\vert $, and if $K\left( \sigma \left( \pm
\xi \right) \right) $ are the Gaussian curvatures at the points $%
\sigma \left( \pm \xi \right) $, then %
\begin{align*}
a\left( \xi \right) &=\left( 2\pi i\right) ^{-1}\exp \left( -2\pi i\sigma
\left( -\xi \right) \cdot \xi -\pi i\left( d-1\right) /4\right) K\left(
\sigma \left( -\xi \right) \right) ^{-1/2} \\
& \ -\left( 2\pi i\right) ^{-1}\exp \left( -2\pi i\sigma \left( \xi \right)
\cdot \xi +\pi i\left( d-1\right) /4\right) K\left( \sigma \left( \xi
\right) \right) ^{-1/2}.
\end{align*}

The remainder $E\left( \xi \right) $ satisfies the
estimates %
\[
\left\vert E\left( \xi \right) \right\vert +\left\vert \nabla E\left( \xi
\right) \right\vert \leq C\left\vert \xi \right\vert ^{-\left( d+3\right)
/2}.
\]
\end{lemma}

\begin{proof} This is a classical result. See \cite{H1, H2, Hl}, or \cite[Corollary 7.7.15]{Ho}, or \cite[Chapter VIII]{St}.
 In particular, as shown before, the lemma for a ball follows straightly from the asymptotic
expansion of Bessel functions. Anyhow, since in most references the exact
constants in this asymptotic expansion are not explicit and a control on the
derivative of the remainder is omitted, it may be helpful to recall a proof.
Write $\xi =\rho \vartheta $, with $\rho>0$ and $|\vartheta|=1$, and denote by $n\left( x\right) $ the outward
unit normal to the boundary at the point $x$. By the divergence theorem, 
\[
\int_{\Omega }\exp \left( -2\pi i\rho \vartheta \cdot x\right)
dx=-\left( 2\pi i\rho \right) ^{-1}\int_{\partial \Omega %
}\vartheta \cdot n\left( x\right) \exp \left( -2\pi i\rho \vartheta \cdot
x\right) dx.
\]

In the surface integral the phase $\vartheta \cdot x$ is stationary at the
points $\sigma \left( \pm \vartheta \right) $ with normals $\pm \vartheta $,
and one can isolate these points with a smooth cutoff $\varphi \left( s\right) $%
, with $\varphi \left( s\right) =0$ if $s\leq 1-2\varepsilon $ and $\varphi \left(
s\right) =1$ if $s\geq 1-\varepsilon $ for some small positive $\varepsilon $, 
\begin{align*}
&\int_{\partial\Omega }\vartheta \cdot n\left( x\right) \exp
\left( -2\pi i\rho \vartheta \cdot x\right) dx \\
&=\int_{\partial \Omega }\varphi \left( \vartheta \cdot n\left(
x\right) \right) \vartheta \cdot n\left( x\right) \exp \left( -2\pi i\rho
\vartheta \cdot x\right) dx \\
& \ +\int_{\partial \Omega }\varphi \left( -\vartheta \cdot n\left(
x\right) \right) \vartheta \cdot n\left( x\right) \exp \left( -2\pi i\rho
\vartheta \cdot x\right) dx \\
& \ +\int_{\partial \Omega }\left( 1-\varphi \left( \vartheta \cdot
n\left( x\right) \right) -\varphi \left( -\vartheta \cdot n\left( x\right)
\right) \right) \vartheta \cdot n\left( x\right) \exp \left( -2\pi i\rho
\vartheta \cdot x\right) dx.
\end{align*}

Since in the domain of integration of the third integral there are no
critical points, this integral decays faster than any power $\rho ^{-N}$
when $\rho \rightarrow +\infty $, and the same is true for the derivatives
of this integral. The first and second integrals are similar to each other. Let us
consider the first one. By a suitable choice of the coordinates $x=\sigma
\left( \vartheta \right) +\left( y,z\right) $, with $y\in \mathbb{R}^{d-1}$
and $z\in \mathbb{R}$, one can move the singular point of the phase to the
origin, and one can assume that in a neighborhood of the origin the boundary 
$\partial \Omega $ is the graph $z=\Phi \left( y\right) $ and the unit
normal at the origin is $\left( 0,-1\right) $. In particular, $%
\nabla \Phi \left( 0\right) =0$. Then, setting $\left( 0,-1\right) \cdot
n\left( x\right) =n\left( y\right) $, one obtains 
\begin{align*}
&\int_{\partial \Omega }\varphi \left( \vartheta \cdot n\left(
x\right) \right) \vartheta \cdot n\left( x\right) \exp \left( -2\pi i\rho
\vartheta \cdot x\right) dx \\
&=\exp \left( -2\pi i\rho \sigma \left( \vartheta \right) \cdot \vartheta
\right) \int_{\mathbb{R}^{d-1}}\varphi \left( n\left( y\right) \right) n\left(
y\right) \exp \left( 2\pi i\rho \Phi \left( y\right) \right) \sqrt{%
1+\left\vert \nabla \Phi \left( y\right) \right\vert ^{2}}dy.
\end{align*}

By \cite[Proposition 6, Chapter VIII, \S 2]{St}, if $\left\{ \mu
_{k}\right\} _{k=1}^{d-1}$ are the eigenvalues of the Hessian matrix $\left[
\partial \Phi \left( y\right) /\partial y_{i}\partial y_{j}\right] $ at the
point $y=0$, then 
\begin{align*}
&\int_{\mathbb{R}^{d-1}}\varphi \left( n\left( y\right) \right) n\left(
y\right) \exp \left( 2\pi i\rho \Phi \left( y\right) \right) \sqrt{%
1+\left\vert \nabla \Phi \left( y\right) \right\vert ^{2}}dy \\
&=\rho ^{-\left( d-1\right) /2}\prod_{k=1}^{d-1}\left( -i\mu _{k}\right)
^{-1/2}+\mathcal{O}\left( \rho ^{-\left( d+1\right) /2}\right) .
\end{align*}

The eigenvalues of the Hessian matrix are the principal curvatures of $%
\partial \Omega $ at $\sigma \left( \vartheta \right) $, and the product of
these eigenvalues is the Gaussian curvature, 
\[
\prod_{k=1}^{d-1}\left( -i\mu _{k}\right) ^{-1/2}=\exp \left( \left(
d-1\right) \pi i/4\right) K\left( \sigma \left( \vartheta \right) \right)
^{-1/2}.
\]

Hence, 
\begin{align*}
&-\left( 2\pi i\rho \right) ^{-1}\int_{\partial \Omega }\varphi \left(
\vartheta \cdot n\left( x\right) \right) \vartheta \cdot n\left( x\right)
\exp \left( -2\pi i\rho \vartheta \cdot x\right) dx \\
&=-\left( 2\pi i\right) ^{-1}\exp \left( -2\pi i\sigma \left( \xi \right)
\cdot \xi +\left( d-1\right) \pi i/4\right) K\left( \sigma \left( \xi
\right) \right) ^{-1/2}\left\vert \xi \right\vert ^{-\left( d+1\right) /2} \\
& \ +\mathcal{O}\left( \left\vert \xi \right\vert ^{-\left( d+3\right)
/2}\right) .
\end{align*}

In order to obtain the main term in the asymptotic expansion one has to sum
the contribution of the point $\sigma \left( \vartheta \right) $ with the
one of the antipodal point $\sigma \left( -\vartheta \right) $. In this way
one obtains the decomposition 
\[
\widehat{\chi }_{\Omega }\left( \xi \right) =a\left( \xi \right) \left\vert
\xi \right\vert ^{-\left( d+1\right) /2}+E\left( \xi \right) .
\]

The remainder has the property $\left\vert E\left( \xi \right) \right\vert
\leq C\left\vert \xi \right\vert ^{-\left( d+3\right) /2}$ as $\left\vert
\xi \right\vert \rightarrow \infty $. Since $\widehat{\chi }_{\Omega }\left(
\xi \right) $ is an entire function of finite exponential type, the above
equality can be differentiated and one obtains 
\[
\nabla \widehat{\chi }_{\Omega }\left( \xi \right) =\left\vert \xi
\right\vert ^{-(d+1)/2}\nabla a\left( \xi \right) -
(({d+1})/2) a\left( \xi \right) \left\vert \xi \right\vert ^{-\left(
d+5\right) /2}\xi +\nabla E\left( \xi \right) .
\]

This is the same as 
\[
\nabla E\left( \xi \right) =\nabla \widehat{\chi }_{\Omega }\left( \xi
\right) -\left\vert \xi \right\vert ^{-\left( d+1\right) /2}\nabla a\left(
\xi \right) +(\left( d+1\right) /2)a\left( \xi \right) \left\vert \xi
\right\vert ^{-\left( d+5\right) /2}\xi .
\]

The term $(\left( d+1\right) /2)a\left( \xi \right) \left\vert \xi \right\vert
^{-\left( d+5\right) /2}\xi $ is $\mathcal{O}\left( \left\vert \xi
\right\vert ^{-\left( d+3\right) /2}\right) $, and both terms $\nabla 
\widehat{\chi }_{\Omega }\left( \xi \right) $ and $\left\vert \xi
\right\vert ^{-\left( d+1\right) /2}\nabla a\left( \xi \right) $ are $%
\mathcal{O}\left( \left\vert \xi \right\vert ^{-\left( d+1\right) /2}\right) 
$, but the main parts of these last terms are the same and they cancel, and
what is left is $\mathcal{O}\left( \left\vert \xi \right\vert ^{-\left(
d+3\right) /2}\right) $. Let us first identify the main part of $\left\vert
\xi \right\vert ^{-\left( d+1\right) /2}\nabla a\left( \xi \right) $ that
comes from the point $\sigma \left( \vartheta \right) $. Recall that $\sigma
\left( \xi \right) \cdot \xi =\sup_{x\in \Omega }\left\{ x\cdot \xi \right\} 
$, the support function of the convex body, has gradient $\nabla \left(
\sigma \left( \xi \right) \cdot \xi \right) =\sigma \left( \xi \right) $.
See \cite{BF}, or \cite[Corollary 1.7.3]{Sc}. Hence, 
\begin{align*}
&\nabla \left( -\left( 2\pi i\right) ^{-1}\exp \left( -2\pi i\sigma \left(
\xi \right) \cdot \xi +\left( d-1\right) \pi i/4\right) K\left( \sigma
\left( \xi \right) \right) ^{-1/2}\right) \\
&=-\left( 2\pi i\right) ^{-1}\exp \left( -2\pi i\sigma \left( \xi \right)
\cdot \xi +\left( d-1\right) \pi i/4\right) \nabla \left( K\left( \sigma
\left( \xi \right) \right) ^{-1/2}\right) \\
& \ +\exp \left( -2\pi i\sigma \left( \xi \right) \cdot \xi +\left( d-1\right)
\pi i/4\right) K\left( \sigma \left( \xi \right) \right) ^{-1/2}\sigma
\left( \xi \right) .
\end{align*}

Since $\sigma \left( \xi \right) $ is homogeneous of degree $0$, $\nabla
\left( K\left( \sigma \left( \xi \right) \right) ^{-1/2}\right) $ is
homogeneous of degree $-1$, so that the main contribution to $\left\vert \xi
\right\vert ^{-\left( d+1\right) /2}\nabla a\left( \xi \right) $ that comes
from the point $\sigma \left( \vartheta \right) $ is 
\[
\exp \left( -2\pi i\sigma \left( \xi \right) \cdot \xi +\left( d-1\right)
\pi i/4\right) K\left( \sigma \left( \xi \right) \right) ^{-1/2}\left\vert
\xi \right\vert ^{-\left( d+1\right) /2}\sigma \left( \xi \right) .
\]

Let us now identify the main part of $\nabla \widehat{\chi }_{\Omega }\left(
\xi \right) $ that comes from the point $\sigma \left( \vartheta \right) $.
The gradient $\nabla \widehat{\chi }_{\Omega }\left( \xi \right) $ is
defined by an integral similar to the one that defines $\widehat{\chi }%
_{\Omega }\left( \xi \right) $, and it has a similar asymptotic expansion, 
\begin{align*}
&\nabla \left( \int_{\Omega }\exp \left( -2\pi i\xi \cdot x\right)
dx\right) =-2\pi i\int_{\Omega }x\exp \left( -2\pi i\xi \cdot
x\right) dx \\
&=-\left\vert \xi \right\vert ^{-2}\xi \int_{\Omega }\exp \left(
-2\pi i\xi \cdot x\right) dx+\left\vert \xi \right\vert ^{-2}\int_{\partial 
\Omega }x\exp \left( -2\pi i\xi \cdot x\right) \xi \cdot n\left(
x\right) dx.
\end{align*}

The first integral is similar to the previous one, but the factor $\left\vert \xi \right\vert ^{-2}\xi $ gives an extra decay, 
\[
\left\vert \left\vert \xi \right\vert ^{-2}\xi \int_{\Omega }x\exp
\left( -2\pi i\xi \cdot x\right) dx\right\vert \leq C\left\vert \xi
\right\vert ^{-\left( d+3\right) /2}.
\]

Arguing as before and isolating the critical point $\sigma \left( \vartheta
\right) $, with the change of variables $x=\sigma \left( \vartheta \right)
+\left( y,z\right) $ one obtains 
\begin{align*}
&\rho ^{-1}\int_{\partial \Omega }x\varphi \left( \vartheta \cdot
n\left( x\right) \right) \vartheta \cdot n\left( x\right) \exp \left( -2\pi
i\rho \vartheta \cdot x\right) dx= \\
&\rho ^{-1}\exp \left( -2\pi i\rho \sigma \left( \vartheta \right) \cdot
\vartheta \right) \int_{\mathbb{R}^{d-1}}\left( y,\Phi \left( y\right)
\right) \varphi \left( n\left( y\right) \right) n\left( y\right) \exp \left(
2\pi i\rho \Phi \left( y\right) \right) \sqrt{1+\left\vert \nabla \Phi
\left( y\right) \right\vert ^{2}}dy \\
&+\rho ^{-1}\exp \left( -2\pi i\rho \sigma \left( \vartheta \right) \cdot
\vartheta \right) \sigma \left( \vartheta \right) \int_{\mathbb{R}%
^{d-1}}\varphi \left( n\left( y\right) \right) n\left( y\right) \exp \left(
2\pi i\rho \Phi \left( y\right) \right) \sqrt{1+\left\vert \nabla \Phi
\left( y\right) \right\vert ^{2}}dy.
\end{align*}

In the first integral the factor $\left( y,\Phi \left( y\right) \right) $
vanishes at the singular point $y=0$ of the phase, and this implies that 
\begin{align*}
&\left\vert \rho ^{-1}\exp \left( -2\pi i\rho \sigma \left( \vartheta \right)
\cdot \vartheta \right) \int_{\mathbb{R}^{d-1}}\left( y,\Phi \left(
y\right) \right) \varphi \left( n\left( y\right) \right) n\left( y\right) \exp
\left( 2\pi i\rho \Phi \left( y\right) \right) \sqrt{1+\left\vert \nabla
\Phi \left( y\right) \right\vert ^{2}}dy\right\vert \\
&\leq C\rho ^{-\left( d+3\right) /2}.
\end{align*}

The second integral is exactly the same that appears in the computation of $\widehat{\chi }_{\Omega }\left( \xi \right) $, 
\begin{align*}
&\rho ^{-1}\exp \left( -2\pi i\rho \sigma \left( \vartheta \right) \cdot
\vartheta \right) \sigma \left( \vartheta \right) \int_{\mathbb{R}%
^{d-1}}\varphi \left( n\left( y\right) \right) n\left( y\right) \exp \left(
2\pi i\rho \Phi \left( y\right) \right) \sqrt{1+\left\vert \nabla \Phi
\left( y\right) \right\vert ^{2}}dy \\
&=\exp \left( -2\pi i\rho \sigma \left( \vartheta \right) \cdot \vartheta
+\left( d-1\right) \pi i/4\right) K\left( \sigma \left( \vartheta \right)
\right) ^{-1/2}\rho ^{-\left( d+1\right) /2}\sigma \left( \vartheta \right) +%
\mathcal{O}\left( \rho ^{-\left( d+3\right) /2}\right) .
\end{align*}

In conclusion, the main parts of $\nabla \widehat{\chi }_{\Omega }\left( \xi
\right) $ and $\left\vert \xi \right\vert ^{-\left( d+1\right) /2}\nabla
a\left( \xi \right) $ cancel, and all that is left is $\mathcal{O}\left(
\left\vert \xi \right\vert ^{-\left( d+3\right) /2}\right) $.
\end{proof}

\begin{lemma}\label{L3}
The Fourier transform of the annulus $\Omega
\left( r,t\right) =\left( r+t/2\right) \Omega \setminus \left( r-t/2\right) \Omega$ can be decomposed into
\[
\widehat{\chi }_{\Omega \left( r,t\right) }\left( \xi \right) =A\left(
r,t,\xi \right) +B\left( r,t,\xi \right) .
\]
The main term is 
\begin{align*}
& A\left( r,t,\xi \right) = \\
&  -\pi ^{-1}r^{\left( d-1\right) /2}\left\vert \xi \right\vert ^{-\left(
d+1\right) /2}\exp \left( -2\pi ir\sigma \left( -\xi \right) \cdot \xi -\pi
i\left( d-1\right) /4\right) K\left( \sigma \left( -\xi \right) \right)
^{-1/2}\sin \left( \pi t\sigma \left( -\xi \right) \cdot \xi \right) \\
& +\pi ^{-1}r^{\left( d-1\right) /2}\left\vert \xi \right\vert ^{-\left(
d+1\right) /2}\exp \left( -2\pi ir\sigma \left( \xi \right) \cdot \xi +\pi
i\left( d-1\right) /4\right) K\left( \sigma \left( \xi \right) \right)
^{-1/2}\sin \left( \pi t\sigma \left( \xi \right) \cdot \xi \right) .
\end{align*}
The remainder has the property that there exists $C>0$ such that for every $r|\xi|\ge 1$ and for every $0< t \leq r$, 
\[
\left\vert B\left( r,t,\xi \right) \right\vert \leq Cr^{\left( d-3\right)
/2}t\left\vert \xi \right\vert ^{-\left( d+1\right) /2}.
\]
\end{lemma}

\begin{proof} With the notation of the previous lemma $\widehat{\chi }_{\Omega }\left( \xi
\right) =a\left( \xi \right) \left\vert \xi \right\vert ^{-\left( d+1\right)
/2}+E\left( \xi \right) $, 
\begin{align*}
\widehat{\chi }_{\Omega \left( r,t\right) }\left( \xi \right)& =\left(
r+t/2\right) ^{d}\widehat{\chi }_{\Omega }\left( \left( r+t/2\right) \xi
\right) -\left( r-t/2\right) ^{d}\widehat{\chi }_{\Omega }\left( \left(
r-t/2\right) \xi \right) \\
&=r^{\left( d-1\right) /2}\left( a\left( \left( r+t/2\right) \xi \right)
-a\left( \left( r-t/2\right) \xi \right) \right) \left\vert \xi \right\vert
^{-\left( d+1\right) /2} \\
& \ +\left( \left( r+t/2\right) ^{\left( d-1\right) /2}-r^{\left( d-1\right)
/2}\right) a\left( \left( r+t/2\right) \xi \right) \left\vert \xi
\right\vert ^{-\left( d+1\right) /2} \\
& \ -\left( \left( r-t/2\right) ^{\left( d-1\right) /2}-r^{\left( d-1\right)
/2}\right) a\left( \left( r-t/2\right) \xi \right) \left\vert \xi
\right\vert ^{-\left( d+1\right) /2} \\
& \ +\left( r+t/2\right) ^{d}\left( E\left( \left( r+t/2\right) \xi \right)
-E\left( \left( r-t/2\right) \xi \right) \right) \\
& \ +\left( \left( r+t/2\right) ^{d}-\left( r-t/2\right) ^{d}\right) E\left(
\left( r-t/2\right) \xi \right) .
\end{align*}

The estimates on $E\left( \xi \right) $ and on $\nabla E\left( \xi \right) $
give 
\begin{align*}
\left\vert \left( \left( r+t/2\right) ^{d}-\left( r-t/2\right) ^{d}\right)
E\left( \left( r-t/2\right) \xi \right) \right\vert 
&\leq Cr^{\left( d-5\right) /2}t\left\vert \xi \right\vert ^{-\left(
d+3\right) /2},\\
\left\vert \left( r+t/2\right) ^{d}\left( E\left( \left( r+t/2\right) \xi
\right) -E\left( \left( r-t/2\right) \xi \right) \right) \right\vert  
&\leq Cr^{\left( d-3\right) /2}t\left\vert \xi \right\vert ^{-\left(
d+1\right) /2}.
\end{align*}

Similarly, one also has 
\[
\left\vert \left( \left( r\pm t/2\right) ^{\left( d-1\right)
/2}-r^{\left( d-1\right) /2}\right) a\left( \left( r\pm t/2\right) \xi
\right) \left\vert \xi \right\vert ^{-\left( d+1\right) /2}\right\vert  
\leq Cr^{\left( d-3\right) /2}t\left\vert \xi \right\vert ^{-\left(
d+1\right) /2}.
\]

The main term comes from $a\left( \left( r+t/2\right) \xi \right) -a\left(
\left( r-t/2\right) \xi \right) $, and it needs a slightly more precise
analysis. Since $\sigma \left( \pm \xi \right) $ is homogeneous of degree
zero, one has $\sigma \left( \pm \left( r\pm t/2\right) \xi \right) =\sigma \left(
\pm \xi \right) $, and a little computation gives 
\begin{align*}
&a\left( \left( r+t/2\right) \xi \right) -a\left( \left( r-t/2\right) \xi
\right) = \\
&-\pi ^{-1}\exp \left( -2\pi ir\sigma \left( -\xi \right) \cdot \xi -\pi
i\left( d-1\right) /4\right) K\left( \sigma \left( -\xi \right) \right)
^{-1/2}\sin \left( \pi t\sigma \left( -\xi \right) \cdot \xi \right)  \\
&+\pi ^{-1}\exp \left( -2\pi ir\sigma \left( \xi \right) \cdot \xi +\pi
i\left( d-1\right) /4\right) K\left( \sigma \left( \xi \right) \right)
^{-1/2}\sin \left( \pi t\sigma \left( \xi \right) \cdot \xi \right) .
\end{align*}
\end{proof}

At this point one can already show that the variance is bounded up to a
constant by the mean. Indeed, it follows from the above lemma that if $t\leq r$ and $r|\xi| \geq 1$, then
\[
\left\vert \widehat{\chi }_{\Omega \left( r,t\right) }\left( \xi \right)
\right\vert \leq Cr^{\left( d-1\right) /2}\left\vert \xi \right\vert
^{-\left( d+1\right) /2}\min \left\{ 1,t\left\vert \xi \right\vert \right\} .
\]

Hence, by Parseval's equality, 
\begin{align*}
&\int_{\mathbb{T}^{d}}\left\vert \sum_{k\in \mathbb{Z}^{d}}\chi _{\Omega
\left( r,t\right) -x}(k)-\left\vert \Omega \left( r,t\right) \right\vert
\right\vert ^{2}dx=\sum_{n\in \mathbb{Z}^{d}\setminus\left\{ 0\right\} }\left\vert 
\widehat{\chi }_{\Omega \left( r,t\right) }\left( n\right) \right\vert ^{2}
\\
&\leq Cr^{d-1}t^{2}\sum_{0<\left\vert n\right\vert \leq 1/t}\left\vert
n\right\vert ^{1-d}+Cr^{d-1}\sum_{1/t<\left\vert n\right\vert <+\infty
}\left\vert n\right\vert ^{-1-d}\leq Cr^{d-1}t.
\end{align*}

Proving an asymptotic estimate of the variance is a more difficult task. One has
to take into account not only the size of the Fourier transform, but also
the oscillations. In particular, the curvature $K\left( x\right) $ and the
support function $\sup_{y\in \Omega }\left\{ x\cdot y\right\} $, which
determine the geometry of the convex body, will play a crucial role.

\begin{lemma}\label{L4}
The variance of the number of integer points in
the shifted annulus can be decomposed into 
\[
\int_{\mathbb{T}^{d}}\left\vert \sum_{k\in \mathbb{Z}^{d}}\chi _{\Omega
\left( r,t\right) -x}(k)-\left\vert \Omega \left( r,t\right) \right\vert
\right\vert ^{2}dx=X\left( r,t\right) +Y\left( r,t\right) +Z\left(
r,t\right),
\]
where
\begin{align*}
X\left( r,t\right) &=2\pi ^{-2}r^{d-1}\sum_{n\in \mathbb{Z}^{d}\setminus \left\{
0\right\} }K\left( \sigma \left( n\right) \right) ^{-1}\sin ^{2}\left( \pi
t\sigma \left( n\right) \cdot n\right) \left\vert n\right\vert ^{-d-1},\\
Y\left( r,t\right) &=-2\pi ^{-2}r^{d-1}\sum_{n\in \mathbb{Z}^{d}\setminus \left\{
0\right\} }\cos \left( 2\pi r\left( \sigma \left( n\right) -\sigma \left(
-n\right) \right) \cdot n-\pi \left( d-1\right) /2\right) \\
&\times K\left( \sigma \left( n\right) \right) ^{-1/2}K\left( \sigma \left(
-n\right) \right) ^{-1/2}\sin \left( \pi t\sigma \left( n\right) \cdot
n\right) \sin \left( \pi t\sigma \left( -n\right) \cdot n\right) \left\vert
n\right\vert ^{-d-1}.
\end{align*}

The remainder $Z\left( r,t\right) $ has the property that
there exists a constant $C$ such if $r\ge 1$ and $t\leq r$ then 
\[
\left\vert Z\left( r,t\right) \right\vert \leq C\left\vert \Omega \left(
r,t\right) \right\vert r^{-1}t\log \left( 2+1/t\right) .
\]
\end{lemma}

\begin{proof} By Lemma \ref{L1} and Lemma \ref{L3}, the variance equals 
\begin{align*}
\sum_{n\in \mathbb{Z}^{d}\setminus\left\{ 0\right\} }\left\vert \widehat{\chi }%
_{\Omega \left( r,t\right) }\left( n\right) \right\vert ^{2}&= 
\sum_{n\in \mathbb{Z}^{d}\setminus\left\{ 0\right\} }A\left( r,t,n\right) \overline{%
A\left( r,t,n\right) }+\sum_{n\in \mathbb{Z}^{d}\setminus\left\{ 0\right\} }A\left(
r,t,n\right) \overline{B\left( r,t,n\right) } \\
&+\sum_{n\in \mathbb{Z}^{d}\setminus\left\{ 0\right\} }B\left( r,t,n\right) 
\overline{A\left( r,t,n\right) }+\sum_{n\in \mathbb{Z}^{d}\setminus\left\{
0\right\} }B\left( r,t,n\right) \overline{B\left( r,t,n\right) }.
\end{align*}

Since $c\left\vert n\right\vert \leq  \sigma \left( n\right) \cdot
n \leq C\left\vert n\right\vert $ for some $C\geq c>0$, Lemma \ref{L3} implies that 
\begin{align*}
\left\vert A\left( r,t,n\right) \right\vert &\leq Cr^{\left( d-1\right)
/2}\left\vert n\right\vert ^{-\left( d+1\right) /2}\min \left\{
1,t\left\vert n\right\vert \right\} ,\\
\left\vert B\left( r,t,n\right) \right\vert &\leq Cr^{\left( d-3\right)
/2}t\left\vert n\right\vert ^{-\left( d+1\right) /2}.
\end{align*}

These estimates give 
\begin{align*}
\sum_{n\in \mathbb{Z}^{d}\setminus\left\{ 0\right\} }\left\vert A\left(
r,t,n\right) \right\vert \left\vert B\left( r,t,n\right) \right\vert 
&\leq Cr^{d-2}t^{2}\sum_{0<\left\vert n\right\vert \leq 1/t}\left\vert
n\right\vert ^{-d}+Cr^{d-2}t\sum_{1/t<\left\vert n\right\vert <+\infty
}\left\vert n\right\vert ^{-d-1}\\
&\leq Cr^{d-2}t^{2}\log \left( 2+1/t\right) ,
\end{align*}
and
\[
\sum_{n\in \mathbb{Z}^{d}\setminus\left\{ 0\right\} }\left\vert B\left(
r,t,n\right) \right\vert ^{2}\leq Cr^{d-3}t^{2}\sum_{n\in \mathbb{Z}%
^{d}\setminus\left\{ 0\right\} }\left\vert n\right\vert ^{-d-1}\leq Cr^{d-3}t^{2}.
\]

The main term is $\sum_{n\in \mathbb{Z}^{d}\setminus\left\{ 0\right\} }\left\vert
A\left( r,t,n\right) \right\vert ^{2}$, and one can check that it is equal
to $X\left( r,t\right) +Y\left( r,t\right) $. 
\end{proof}

It follows from the Cauchy-Schwarz inequality that in the statement of the
above lemma the series $Y\left( r,t\right) $ with cosines is smaller than
the series $X\left( r,t\right) $. 
 Moreover, the cancellations due to the
change of sign of the cosine lead to conjecture that $Y\left( r,t\right) $
is indeed much smaller than $X\left( r,t\right) $, and it gives a negligible
contribution to the variance. Also observe that the single terms in the
expansions $X\left( r,t\right) $ and $Y\left( r,t\right) $ give negligible
contributions to the series. This suggests that these series are asymptotic
to integrals, and at least for $X\left( r,t\right) $ this is the case.

\begin{lemma}\label{L5}
If $\left\vert \Omega \right\vert $ is the volume of the convex body, and with the definition of $X\left( r,t\right) $ in Lemma \ref{L4}, 
we have
\[
X\left( r,t\right) =d\left\vert \Omega \right\vert r^{d-1}t+W\left(
r,t\right) .
\]

The remainder $W\left( r,t\right) $ has the property that for
some $C$ and every $r\ge 1$ and $t\leq r$, 
\[
\left\vert W\left( r,t\right) \right\vert \leq C\left\vert \Omega \left(
r,t\right) \right\vert t\log \left( 2+1/t\right) .
\]
\end{lemma}

\begin{proof} Identifying the torus $\mathbb{T}^{d}=\mathbb{R}^{d}/\mathbb{%
Z}^{d}$ with the unit cube $\left\{ -1/2\leq x_{j}<1/2\right\} $ and
decomposing $\mathbb{R}^{d}$ into $\bigcup_{n\in \mathbb{Z}^{d}}\left\{ 
\mathbb{T}^{d}+n\right\} $, one gets 
\begin{align*}
X(r,t)&=2\pi ^{-2}r^{d-1}\sum_{n\in \mathbb{Z}^{d}\setminus\left\{ 0\right\} }K\left(
\sigma \left( n\right) \right) ^{-1}\sin ^{2}\left( \pi t\sigma \left(
n\right) \cdot n\right) \left\vert n\right\vert ^{-d-1} \\
&=2\pi ^{-2}r^{d-1}\int_{\mathbb{R}^{d}}K\left( \sigma \left( x\right)
\right) ^{-1}\sin ^{2}\left( \pi t\sigma \left( x\right) \cdot x\right)
\left\vert x\right\vert ^{-d-1}dx \\
&-2\pi ^{-2}r^{d-1}\int_{\mathbb{T}^{d}}K\left( \sigma \left( x\right)
\right) ^{-1}\sin ^{2}\left( \pi t\sigma \left( x\right) \cdot x\right)
\left\vert x\right\vert ^{-d-1}dx \\
&-2\pi ^{-2}r^{d-1}\sum_{n\in \mathbb{Z}^{d}\setminus\left\{ 0\right\} }\int_{%
\mathbb{T}^{d}}\left( K\left( \sigma \left( n+x\right) \right) ^{-1}-K\left(
\sigma \left( n\right) \right) ^{-1}\right)\\
& \quad \times \sin ^{2}\left( \pi t\sigma
\left( n+x\right) \cdot \left( n+x\right) \right) \left\vert n+x\right\vert
^{-d-1}dx \\
&-2\pi ^{-2}r^{d-1}\sum_{n\in \mathbb{Z}^{d}\setminus\left\{ 0\right\} }K\left(
\sigma \left( n\right) \right) ^{-1}\int_{\mathbb{T}^{d}}\left\vert n+x\right\vert ^{-d-1}\\
& \quad \times \left( \sin
^{2}\left( \pi t\sigma \left( n+x\right) \cdot \left( n+x\right) \right)
-\sin ^{2}\left( \pi t\sigma \left( n\right) \cdot n\right) \right)
dx \\
&-2\pi ^{-2}r^{d-1}\sum_{n\in \mathbb{Z}^{d}\setminus\left\{ 0\right\} }K\left(
\sigma \left( n\right) \right) ^{-1}\sin ^{2}\left( \pi t\sigma \left(
n\right) \cdot n\right)\\
& \quad \times \int_{\mathbb{T}^{d}}\left( \left\vert
n+x\right\vert ^{-d-1}-\left\vert n\right\vert ^{-d-1}\right) dx.
\end{align*}

First at all, one has 
\begin{align*}
&2\pi ^{-2}r^{d-1}\int_{\mathbb{T}^{d}}K\left( \sigma \left( x\right)
\right) ^{-1}\sin ^{2}\left( \pi t\sigma \left( x\right) \cdot x\right)
\left\vert x\right\vert ^{-d-1}dx \\
&\leq 2r^{d-1}t^{2}\int_{\mathbb{T}^{d}}K\left( \sigma \left( x\right)
\right) ^{-1}\left( \sigma \left( x\right) \cdot x\right) ^{2}\left\vert
x\right\vert ^{-d-1}dx\leq Cr^{d-1}t^{2}.
\end{align*}

Then observe that $\sigma \left( x\right) $ is smooth in $\mathbb{R}%
^{d}\setminus\left\{ 0\right\} $ and homogeneous of degree zero. Moreover, as mentioned before, 
$c\left\vert x\right\vert \leq \sigma \left( x\right) \cdot x\leq C\left\vert
x\right\vert $ for some $C\geq c>0$ and every $x\in \mathbb{R}^{d}$. Hence
also $K\left( \sigma \left( x\right) \right) ^{-1}$ is smooth in $\mathbb{R}%
^{d}\setminus\left\{ 0\right\} $ and homogeneous of degree zero, and for every $x\in 
\mathbb{T}^{d}$ and $n\in \mathbb{Z}^{d}\setminus\left\{ 0\right\} $ one has 
\[
\left\vert K\left( \sigma \left( n+x\right) \right) ^{-1}-K\left( \sigma
\left( n\right) \right) ^{-1}\right\vert \leq C\left\vert n\right\vert ^{-1}.
\]

Hence, 
\begin{align*}
&2\pi ^{-2}r^{d-1}\sum_{n\in \mathbb{Z}^{d}\setminus\left\{ 0\right\} }\int_{%
\mathbb{T}^{d}}\left\vert K\left( \sigma \left( n+x\right) \right)
^{-1}-K\left( \sigma \left( n\right) \right) ^{-1}\right\vert\\
& \ \times \sin^{2}\left( \pi t\sigma \left( n+x\right) \cdot \left( n+x\right) \right)
\left\vert n+x\right\vert ^{-d-1}dx \\
&\leq Cr^{d-1}t^{2}\sum_{0<\left\vert n\right\vert \leq 1/t}\left\vert
n\right\vert ^{-d}+Cr^{d-1}\sum_{1/t<\left\vert n\right\vert <+\infty
}\left\vert n\right\vert ^{-d-2}\\
&\leq Cr^{d-1}t^{2}\log \left( 2+1/t\right) .
\end{align*}

Similarly, by the trigonometric identity $\sin ^{2}\left( x\right) -\sin
^{2}\left( y\right) =\sin \left( x+y\right) \sin \left( x-y\right) $, and since $|\sigma(x)\cdot x-\sigma(y)\cdot y|\leq C|x-y|$,
\begin{align*}
&2\pi ^{-2}r^{d-1}\sum_{n\in \mathbb{Z}^{d}\setminus\left\{ 0\right\} }K\left(
\sigma \left( n\right) \right) ^{-1}\int_{\mathbb{T}^{d}}\left\vert n+x\right\vert ^{-d-1}\\
& \quad \times \left\vert \sin
^{2}\left( \pi t\sigma \left( n+x\right) \cdot \left( n+x\right) \right)
-\sin ^{2}\left( \pi t\sigma \left( n\right) \cdot n\right) \right\vert
dx \\
&\leq 2\pi ^{-2}r^{d-1}\sum_{n\in \mathbb{Z}^{d}\setminus\left\{ 0\right\} }K\left(
\sigma \left( n\right) \right) ^{-1}\int_{\mathbb{T}^{d}}\left\vert \sin
\left( \pi t\left( \sigma \left( n+x\right) \cdot \left( n+x\right) +\sigma
\left( n\right) \cdot n\right) \right) \right\vert \\
& \quad \times \left\vert \sin \left( \pi t\left( \sigma \left( n+x\right) \cdot
\left( n+x\right) -\sigma \left( n\right) \cdot n\right) \right) \right\vert
\left\vert n+x\right\vert ^{-d-1}dx \\
&\leq Cr^{d-1}t^{2}\sum_{0<\left\vert n\right\vert \leq 1/t}\left\vert
n\right\vert ^{-d}+Cr^{d-1}t\sum_{1/t<\left\vert n\right\vert <+\infty
}\left\vert n\right\vert ^{-d-1}\\
&\leq Cr^{d-1}t^{2}\log \left( 2+1/t\right) .
\end{align*}

And the last term is 
\begin{align*}
&2\pi ^{-2}r^{d-1}\sum_{n\in \mathbb{Z}^{d}\setminus\left\{ 0\right\} }K\left(
\sigma \left( n\right) \right) ^{-1}\sin ^{2}\left( \pi t\sigma \left(
n\right) \cdot n\right) \int_{\mathbb{T}^{d}}\left\vert \left\vert
n+x\right\vert ^{-d-1}-\left\vert n\right\vert ^{-d-1}\right\vert dx \\
&\leq Cr^{d-1}t^{2}\sum_{0<n\leq 1/t}\left\vert n\right\vert
^{-d}+Cr^{d-1}\sum_{1/t<n<+\infty }\left\vert n\right\vert ^{-d-2}dx\\
&\leq
Cr^{d-1}t^{2}\log \left( 2+1/t\right) .
\end{align*}

Finally, an integration in polar coordinates $x=\rho \vartheta $ with a
change of variables gives 
\begin{align*}
&2\pi ^{-2}r^{d-1}\int_{\mathbb{R}^{d}}K\left( \sigma \left( x\right)
\right) ^{-1}\sin ^{2}\left( \pi t\sigma \left( x\right) \cdot x\right)
\left\vert x\right\vert ^{-d-1}dx \\
&=2\pi ^{-2}r^{d-1}\int_{0}^{+\infty }\int_{\left\{ \left\vert \vartheta
\right\vert =1\right\} }K\left( \sigma \left( \vartheta \right) \right)
^{-1}\sin ^{2}\left( \pi t\rho \sigma \left( \vartheta \right) \cdot
\vartheta \right) \rho ^{-2}d\rho d\vartheta \\
&=2\pi ^{-1}r^{d-1}t\left( \int_{0}^{+\infty }\sin ^{2}\left( s\right)
s^{-2}ds\right) \left( \int_{\left\{ \left\vert \vartheta \right\vert
=1\right\} }K\left( \sigma \left( \vartheta \right) \right) ^{-1}\sigma
\left( \vartheta \right) \cdot \vartheta \ d\vartheta \right) .
\end{align*}

The first integral can be evaluated using residues, 
\[
\int_{0}^{+\infty }\dfrac{\sin ^{2}\left( s\right) }{s^{2}}%
ds=\int_{-\infty }^{+\infty }\dfrac{1-\cos \left( 2s\right) }{4s^{2}}ds=%
\mathrm{Re}\left( \int_{-\infty }^{+\infty }\dfrac{1-\exp \left( 2iz\right) }{%
4z^{2}}dz\right) =\dfrac{\pi }{2}.
\]

The integral with the curvature is $d$ times the volume of the convex body $\Omega $, 
\[
\int_{\left\{ \left\vert \vartheta \right\vert =1\right\} }K\left( \sigma
\left( \vartheta \right) \right) ^{-1}\sigma \left( \vartheta \right) \cdot
\vartheta \ d\vartheta =d\ \left\vert \Omega \right\vert .
\]

This comes from the definition of the curvature as the Jacobian determinant
of the Gauss map. $K\left( \sigma \left( \vartheta \right) \right)
^{-1}d\vartheta =dA$ is an infinitesimal element of surface area of $%
\partial \Omega $, and $\sigma \left( \vartheta \right) \cdot \vartheta \ $%
is the height of the cone with vertex\ $0$ and base $dA$. Hence, 
\begin{equation*}
2\pi ^{-2}r^{d-1}\int_{\mathbb{R}^{d}}K\left( \sigma \left( x\right)
\right) ^{-1}\sin ^{2}\left( \pi t\sigma \left( x\right) \cdot x\right)
\left\vert x\right\vert ^{-d-1}dx=d\left\vert \Omega \right\vert r^{d-1}t.
\end{equation*}
\end{proof}

Observe that the only restriction on the indexes in the above lemmas is $r\ge 1$ and $t\leq r
$, and the assumption $t\leq r^{-\alpha }$ with $\alpha >\left( d-1\right)
/\left( d+1\right) $ in the statement of the theorem has not been used. 
It remains to estimate $Y\left( r,t\right) $, and this is the most delicate
part of the proof. If one assumes that the series that defines $Y\left(
r,t\right) $ is asymptotic to an integral, then one can easily check that
this integral is negligible with respect to $X\left( r,t\right) $. We do not
know under which assumptions the series that defines $Y\left( r,t\right) $
is asymptotic to an integral, as it is the case for $X\left( r,t\right) $.
But, by Remark \ref{R1}, some assumptions are necessary. For this reason we need to
follow a more circuitous path. By the Cauchy--Schwarz inequality, $\left\vert
Y\left( r,t\right) \right\vert \leq X\left( r,t\right) \leq Cr^{d-1}t$. In
order to obtain some better estimates one has to take into account the
cancellations in the series that defines $Y\left( r,t\right) $. We need a
couple of preliminary lemmas.

\begin{lemma}\label{L6}
If $X$ and $Y$ are two convex
bodies with smooth boundaries with everywhere positive Gaussian curvature,
then also the Minkowski sum $X+Y$, that is the set obtained by
adding each vector in $X$ to each vector in $Y$, is a
convex body with smooth boundary with everywhere positive curvature.
\end{lemma}

\begin{proof} The fact that $X+Y$ has smooth boundary is proved in \cite{KP}. 
The fact that the boundary has positive Gaussian
curvature can be seen as follows. The strict convexity of $X$ and $Y$
implies that for every $z$ on the boundary $\partial \left( X+Y\right) $
there exist only one $x\in \partial X$ and one $y\in \partial Y$ with $z=x+y$%
. The curvature assumption implies that there exist balls $B_{x}$ and $B_{y}$
with $x\in \partial B_{x}$, $X\subseteq B_{x}$, $y\in \partial B_{y}$, $%
Y\subseteq B_{y}$. It follows that $x+y\in \partial \left(
B_{x}+B_{y}\right) $ and $X+Y\subseteq B_{x}+B_{y}$. Hence the curvature of $%
\partial \left( X+Y\right) $ at the point $x+y$ is at least as large as the
curvature of $B_{x}+B_{y}$, which is a ball with radius the sum of the radii
of $B_{x}$ and $B_{y}$. By the way, without the
curvature assumption the smoothness of the Minkowsky sum may fail. Indeed it has been proved in \cite{Ki}
that there exist convex sets in the plane with real analytic
boundaries, but with the smoothness of the sum not exceeding $C^{20/3}$. And
if the boundaries are only $C^{\infty }$ then the smoothness of the sum may
break out at the level $C^{5}$.
\end{proof}

\begin{lemma}\label{L7}
Denote by $\sigma \left( \pm x\right) $
the points of the boundary $\partial \Omega $ with outward unit
normals $\pm x/\left\vert x\right\vert $, and define 
\begin{equation*}
\zeta \left( x\right) =\left( \sigma \left( x\right) -\sigma \left(
-x\right) \right) \cdot x.
\end{equation*}
Also denote by $A=\Omega +\left( -\Omega \right) $ the
Minkowski sum of $\Omega $ and $-\Omega $. Finally,
assume that $\psi \left( x\right) $ is a smooth function in 
$\mathbb{R}^{d}$ with support in $\varepsilon \leq \left\vert
x\right\vert \leq 1/\varepsilon $, and such that for some $\eta $
 and for every multi index $k$,
\begin{equation*}
\left\vert \dfrac{\partial ^{k}}{\partial x^{k}}\psi \left( s\right)
\right\vert \leq C\left( k\right) \varepsilon ^{-\eta -\left\vert
k\right\vert }.
\end{equation*}
Then for every $j>0$ there exist positive constants $C$
and $\gamma $, such that for every $\xi $ in 
$\mathbb{R}^{d}$, every $\lambda >0$, and every 
$0<\varepsilon <1$, one has 
\begin{align*}
&\left\vert \int_{\mathbb{R}^{d}}\psi \left( x\right) \exp \left( 2\pi
i\lambda \left( \zeta \left( x\right) -\xi \cdot x\right) \right)
dx\right\vert \\
&\leq C\varepsilon ^{-\gamma }\min \left\{ \lambda ^{-\left( d-1\right) /2},\
\left( \lambda \, \mathrm{distance}\left\{ \xi ,\ \partial A\right\} \right)
^{-j}\right\} .
\end{align*}
\end{lemma}

\begin{proof} Recall that $\sigma \left( x\right) \cdot x=\sup_{y\in
\Omega }\left\{ y\cdot x\right\} $, the support function of the convex body,
has gradient $\nabla \left( \sigma \left( x\right) \cdot x\right) =\sigma
\left( x\right) $. See \cite{BF}, or \cite[Corollary 1.7.3]{Sc}. Also observe that when $x$ varies in $\mathbb{R}^{d}\setminus\left\{
0\right\} $, then $\sigma \left( x\right) -\sigma \left( -x\right) $
describes the boundary of $A=\Omega +\left( -\Omega \right) $. Hence, 
\begin{equation*}
\left\vert \nabla \left( \left( \sigma \left( x\right) -\sigma \left(
-x\right) \right) \cdot x-\xi \cdot x\right) \right\vert =\left\vert \left(
\sigma \left( x\right) -\sigma \left( -x\right) \right) -\xi \right\vert
\geq \mathrm{distance}\left\{ \xi ,\ \partial A\right\} .
\end{equation*}

Then a repeated integration by parts gives 
\begin{equation*}
\left\vert \int_{\mathbb{R}^{d}}\psi \left( x\right) \exp \left( 2\pi
i\lambda \left( \zeta \left( x\right) -\xi \cdot x\right) \right)
dx\right\vert \leq C\varepsilon ^{-\gamma }\left( \lambda \, \mathrm{distance}%
\left\{ \xi ,\ \partial A\right\} \right) ^{-j}.
\end{equation*}

See e.g. \cite[Chapter VIII, \S 2.1]{St}. This proves half of the lemma. In
order to complete the proof, observe that the function $\left( \sigma \left(
x\right) -\sigma \left( -x\right) \right) \cdot x$ is the support function
of $A=\Omega +\left( -\Omega \right) $, and recall that, by the previous
lemma, the boundary of this body is smooth with everywhere positive Gaussian
curvature. It follows that this support function is homogeneous of degree
one, and that one eigenvalue of the Hessian matrix is zero, but all other eigenvalues
are positive. See \cite[Corollary 2.5.2]{Sc}. Hence, the Hessian of the
phase $\zeta \left( x\right) -\xi \cdot x$, which is the Hessian of $\left(
\sigma \left( x\right) -\sigma \left( -x\right) \right) \cdot x$, has rank $%
d-1$, and it follows that 
\begin{equation*}
\left\vert \int_{\mathbb{R}^{d}}\psi \left( x\right) \exp \left( 2\pi
i\lambda \left( \zeta \left( x\right) -\xi \cdot x\right) \right)
dx\right\vert \leq C\varepsilon ^{-\gamma }\lambda ^{-\left( d-1\right) /2}.
\end{equation*}

In order to see this, it suffices to apply the coarea formula to the level
set of the function $\zeta \left( x\right) $. Then one ends up to estimate
the Fourier transform of a smooth measure carried by a smooth surface with
everywhere positive Gaussian curvature. See e.g. \cite{Li}, or \cite[Chapter VIII,\S 2.3 and \S 3.1]{St}.
\end{proof}

\begin{lemma}\label{L8}
With the definition of $Y\left( r,t\right) $
in Lemma \ref{L4}, if $\alpha >\left( d-1\right) /\left( d+1\right) $
 there exist positive constants $C$ and $\beta $ such that for every 
$1\leq r<+\infty $ and every $0<t\leq
r^{-\alpha }$ one has, 
\begin{equation*}
\left\vert Y\left( r,t\right) \right\vert \leq C\left\vert \Omega \left(
r,t\right) \right\vert t^{\beta }.
\end{equation*}
\end{lemma}

\begin{proof} In order to simplify the notation, set 
\begin{align*}
\vartheta&=\pi \left( d-1\right) /2 ,\\
\zeta\left( x\right)&=\left( \sigma \left( x\right) -\sigma \left( -x\right) \right) \cdot x,\\
\varphi \left( x\right)&=K\left( \sigma \left( x\right) \right) ^{-1/2}K\left( \sigma \left(
-x\right) \right) ^{-1/2}\sin \left( \pi \sigma \left( x\right) \cdot
x\right) \sin \left( \pi \sigma \left( -x\right) \cdot x\right) \left\vert
x\right\vert ^{-d-1} .
\end{align*}

Then one can rewrite the series that defines $Y\left( r,t\right) $ as 
\begin{equation*}
Y\left( r,t\right) =-2\pi ^{-2}r^{d-1}t\sum_{n\in \mathbb{Z}^{d}\setminus\left\{
0\right\} }t^{d}\varphi \left( tn\right) \cos \left( 2\pi rt^{-1}\zeta
\left( tn\right) -\vartheta \right) .
\end{equation*}

Observe that the factor $r^{d-1}t$ in front of the series is of the order of 
$\left\vert \Omega \left( r,t\right) \right\vert $. Hence, in order to prove
the lemma it suffices to show that the series is bounded by $Ct^{\beta }$
when $t\leq r^{-\alpha }$. Let $0<\varepsilon <1/2$ and let $\chi \left(
s\right) $ be a smooth function with support in $\varepsilon \leq s\leq
1/\varepsilon $, with $0\leq \chi \left( s\right) \leq 1$ and equal to 1 in $%
2\varepsilon \leq s\leq 1/2\varepsilon $, and with 
\begin{equation*}
\left\vert \dfrac{d^{j}}{ds^{j}}\chi \left( s\right) \right\vert \leq
C\varepsilon ^{-j}.
\end{equation*}

With this cut off function, one can decompose 
\begin{align*}
&\sum_{n\in \mathbb{Z}^{d}\setminus\left\{ 0\right\} }t^{d}\varphi \left( tn\right)
\cos \left( 2\pi rt^{-1}\zeta \left( tn\right) -\vartheta \right) \\
&= \sum_{n\in \mathbb{Z}^{d}\setminus\left\{ 0\right\} }t^{d}\left( 1-\chi \left(
\left\vert tn\right\vert \right) \right) \varphi \left( tn\right) \cos
\left( 2\pi rt^{-1}\zeta \left( tn\right) -\vartheta \right) \\
& \ +\sum_{n\in \mathbb{Z}^{d}\setminus\left\{ 0\right\} }t^{d}\chi \left( \left\vert
tn\right\vert \right) \varphi \left( tn\right) \cos \left( 2\pi rt^{-1}\zeta
\left( tn\right) -\vartheta \right) .
\end{align*}

One has 
\begin{align*}
&\left\vert \sum_{n\in \mathbb{Z}^{d}\setminus\left\{ 0\right\} }t^{d}\left( 1-\chi
\left( \left\vert tn\right\vert \right) \right) \varphi \left( tn\right)
\cos \left( 2\pi rt^{-1}\zeta \left( tn\right) -\vartheta \right) \right\vert\\
&\leq \pi ^{2}\sup \left\{ \left\vert \sigma \left( n\right) \right\vert
^{2}K\left( \sigma \left( n\right) \right) ^{-1}\right\}
t\sum_{0<\left\vert n\right\vert <2\varepsilon /t}\left\vert n\right\vert
^{1-d} \\
& \ +\sup \left\{ K\left( \sigma \left( n\right) \right) ^{-1}\right\}
t^{-1}\sum_{1/\left( 2\varepsilon t\right) <\left\vert n\right\vert
<+\infty }\left\vert n\right\vert ^{-d-1} \\
&\leq C\varepsilon .
\end{align*}

Again, in order to simplify a bit the notation, set 
\begin{equation*}
f(x)=\chi \left( \left\vert x\right\vert \right) \varphi \left( x\right) \cos
\left( 2\pi rt^{-1}\zeta \left( x\right) -\vartheta \right)  .
\end{equation*}

Then, if $\widehat{f}\left( \xi \right) =\int_{\mathbb{R}^{d}}f\left(
x\right) \exp \left( -2\pi i\xi \cdot x\right) dx$ is the Fourier transform
of $f\left( x\right) $, the Poisson summation formula with a change of
variables gives 
\begin{equation*}
\sum_{n\in \mathbb{Z}^{d}\setminus\left\{ 0\right\} }t^{d}\chi \left( \left\vert
tn\right\vert \right) \varphi \left( tn\right) \cos \left( 2\pi rt^{-1}\zeta
\left( tn\right) -\vartheta \right) =\sum_{n\in \mathbb{Z}^{d}}t^{d}f\left(
tn\right) =\sum_{n\in \mathbb{Z}^{d}}\widehat{f}\left( t^{-1}n\right) .
\end{equation*}

Observe that the function $f\left( x\right) $ is smooth with compact
support, and that $\widehat{f}\left( \xi \right) $ has fast decay at
infinity. In particular, in the above series there are no
problems of convergence. Writing a cosine as a sum of exponentials, one has 
\begin{align*}
\widehat{f}\left( t^{-1}n\right)& =\int_{\mathbb{R}^{d}}\chi \left(
\left\vert x\right\vert \right) \varphi \left( x\right) \cos \left( 2\pi
rt^{-1}\zeta \left( x\right) -\vartheta \right) \exp \left( -2\pi
it^{-1}n\cdot x\right) dx \\
&=2^{-1}\exp \left( -i\vartheta \right) \int_{\mathbb{R}^{d}}\chi \left(
\left\vert x\right\vert \right) \varphi \left( x\right) \exp \left( 2\pi
irt^{-1}\left( \zeta \left( x\right) -r^{-1}n\cdot x\right) \right) dx \\
& \ +2^{-1}\exp \left( i\vartheta \right) \int_{\mathbb{R}^{d}}\chi \left(
\left\vert x\right\vert \right) \varphi \left( x\right) \exp \left( 2\pi
irt^{-1}\left( -\zeta \left( x\right) -r^{-1}n\cdot x\right) \right) dx.
\end{align*}

Then the previous lemma with $\lambda =rt^{-1}$ and $\xi =\pm r^{-1}n$ gives
for every $j$, 
\begin{equation*}
\left\vert \widehat{f}\left( t^{-1}n\right) \right\vert \leq C\varepsilon
^{-\gamma }\min \left\{ \left( rt^{-1}\right) ^{-\left( d-1\right) /2},\
\left( t^{-1}\mathrm{distance}\left\{  n,\ \partial \left( rA\right)
\right\} \right) ^{-j}\right\},
\end{equation*}
 where the term $\pm n$ in the right-hand side has been replaced by $n$ 
 because $A$ is symmetric.
 
At this point, without pretense of rigor one could conclude the proof as
follows. The above Fourier transform is concentrated in the annulus $\left\{ 
\mathrm{distance}\left\{ n,\ \partial \left( rA\right) \right\} \leq
t\right\} $ which has a measure dominated by $Cr^{d-1}t$, and in this
annulus $\left\vert \widehat{f}\left( t^{-1}n\right) \right\vert \leq
C\varepsilon ^{-\gamma }\left( rt^{-1}\right) ^{-\left( d-1\right) /2}$.
This should imply that 
\begin{equation*}
\sum_{n\in \mathbb{Z}^{d}}\left\vert \widehat{f}\left( t^{-1}n\right)
\right\vert \leq C\varepsilon ^{-\gamma }\left( rt^{-1}\right) ^{-\left(
d-1\right) /2}r^{d-1}t=C\varepsilon ^{-\gamma }r^{\left( d-1\right)
/2}t^{\left( d+1\right) /2}.
\end{equation*}

If $t\leq r^{-\alpha }$ with $\alpha >\left( d-1\right) /\left( d+1\right) $%
, then one can choose $\varepsilon \rightarrow 0+$ such that $\varepsilon
^{-\gamma }r^{\left( d-1\right) /2}t^{\left( d+1\right) /2}\rightarrow 0+$
as $r\rightarrow +\infty $, and this would conclude this pseudo proof. The
proof with full details is a bit more involved. For every $0<s<1$, 
\begin{align*}
\sum_{n\in \mathbb{Z}^{d}}\left\vert \widehat{f}\left( t^{-1}n\right)
\right\vert &
\leq C\varepsilon ^{-\gamma }r^{-\left( d-1\right) /2}t^{\left( d-1\right)
/2}\sum_{\mathrm{distance}\left\{ n,\ \partial \left( rA\right) \right\}
\leq s}1 \\
& +C\varepsilon ^{-\gamma }t^{j}s^{-j}\sum_{\mathrm{distance}\left\{ n,\
\partial \left( rA\right) \right\} \leq 1}1 \\
&+C\varepsilon ^{-\gamma }t^{j}\sum_{k=1}^{+\infty }2^{-jk}\left( \sum_{%
\mathrm{distance}\left\{ n,\ \partial \left( rA\right) \right\} \leq
2^{k}}1\right) .
\end{align*}

In order to estimate the sum over $\left\{ \mathrm{distance}\left\{ n,\
\partial \left( rA\right) \right\} \leq s\right\} $, observe that for some
positive constant $c$ and for every $s<1\leq r$ one has 
\begin{equation*}
\left\{ \mathrm{distance}\left\{ n,\ \partial \left( rA\right) \right\}
\leq s\right\} \subseteq \left( r+cs\right) A\setminus\left( r-cs\right) A.
\end{equation*}

By Lemma \ref{L6} the convex body $A=\Omega +\left( -\Omega \right) $ has a smooth
boundary with everywhere positive Gaussian curvature, and it has been proved
in  \cite{H2, Hl} that there exists a positive constant $C$ such that
for every $r\geq 1$, 
\begin{equation*}
\left\vert \sum_{n\in rA}1-r^{d}\left\vert A\right\vert \right\vert \leq
Cr^{d\left( d-1\right) /\left( d+1\right) }.
\end{equation*}

See also \cite[Theorem 7.7.16]{Ho}. This implies that 
\begin{align*}
&\sum_{\mathrm{distance}\left\{ n,\ \partial \left( rA\right) \right\} \leq
s}1 \\
&\leq \left\vert \sum_{n\in \left( r+cs\right) A}1-\left( r+cs\right)
^{d}\left\vert A\right\vert \right\vert +\left\vert \sum_{n\in \left(
r-cs\right) A}1-\left( r-cs\right) ^{d}\left\vert A\right\vert \right\vert \\
& \ +\left\vert \left( r+cs\right) ^{d}-\left( r-cs\right) ^{d}\right\vert
\left\vert A\right\vert \\
&\leq C\left( r^{d\left( d-1\right) /\left( d+1\right) }+r^{d-1}s\right) .
\end{align*}

The choice $s=r^{-\left( d-1\right) /\left( d+1\right) }$, so that $%
r^{d\left( d-1\right) /\left( d+1\right) }=r^{d-1}s$, then gives 
\begin{equation*}
\varepsilon ^{-\gamma }r^{-\left( d-1\right) /2}t^{\left( d-1\right)
/2}\sum_{\mathrm{distance}\left\{ n,\ \partial \left( rA\right) \right\}
\leq s}1\leq C\varepsilon ^{-\gamma }r^{\left( d-1\right) ^{2}/\left(
2d+2\right) }t^{\left( d-1\right) /2}.
\end{equation*}

In order to estimate the sum over $\left\{ \mathrm{distance}\left\{ n,\
\partial \left( rA\right) \right\} \leq 2^{k}\right\} $, observe that 
\begin{equation*}
\sum_{\mathrm{distance}\left\{ n,\ \partial \left( rA\right) \right\} \leq
2^{k}}1\leq 
\begin{cases}
Cr^{d-1}2^{k} & \text{if }2^{k}\leq r\text{,} \\ 
C2^{dk} & \text{if }2^{k}\geq r\text{.}
\end{cases}
\end{equation*}

It follows that, with the choice $s=r^{-\left( d-1\right) /\left( d+1\right)
}$, 
\begin{equation*}
\varepsilon ^{-\gamma }t^{j}s^{-j}\sum_{\mathrm{distance}\left\{ n,\
\partial \left( rA\right) \right\} \leq 1}1\leq C\varepsilon ^{-\gamma
}r^{d-1+j\left( d-1\right) /\left( d+1\right) }t^{j}.
\end{equation*}

And if $j$ is suitably large it also follows that 
\begin{equation*}
\varepsilon ^{-\gamma }t^{j}\sum_{k=1}^{+\infty }2^{-jk}\left( \sum_{%
\mathrm{distance}\left\{ n,\ \partial \left( rA\right) \right\} \leq
2^{k}}1\right) \leq C\varepsilon ^{-\gamma }r^{d-1}t^{j}.
\end{equation*}

Collecting all these estimates, and assuming that $t\leq r^{-\alpha }$ for
some $\alpha >\left( d-1\right) /\left( d+1\right) $ and that $j$ is
sufficiently large, one obtains that 
\begin{align*}
\sum_{n\in \mathbb{Z}^{d}}\left\vert \widehat{f}\left( t^{-1}n\right)
\right\vert & \leq C\varepsilon ^{-\gamma }\left( r^{\left( d-1\right)
^{2}/\left( 2d+2\right) }t^{\left( d-1\right) /2}+r^{d-1+j\left( d-1\right)
/\left( d+1\right) }t^{j}+r^{d-1}t^{j}\right) \\
&\leq C\varepsilon ^{-\gamma }\left( r^{\left( d-1\right) ^{2}/\left(
2d+2\right) }t^{\left( d-1\right) /2}+r^{d-1+j\left( d-1\right) /\left(
d+1\right) }t^{j}\right) \\
&\leq C\varepsilon ^{-\gamma }r^{\left( d-1\right) ^{2}/\left( 2d+2\right)
}t^{\left( d-1\right) /2}\left( 1+r^{d-1}\left( r^{\left( d-1\right) /\left(
d+1\right) }t\right) ^{j-\left( d-1\right) /2}\right) \\
&\leq C\varepsilon ^{-\gamma }\left( r^{\left( d-1\right) /\left( d+1\right)
}t\right) ^{\left( d-1\right) /2}.
\end{align*}

Assuming again that $t\leq r^{-\alpha }$ for some $\alpha >\left(
d-1\right) /\left( d+1\right) $, and with the choice $\varepsilon =\left(
r^{\left( d-1\right) /\left( d+1\right) }t\right) ^{\left( d-1\right)
/\left( 2\gamma +2\right) }$, one obtains 
\begin{align*}
\left\vert Y\left( r,t\right) \right\vert &\leq C\left\vert \Omega \left(
r,t\right) \right\vert \left( \varepsilon +\varepsilon ^{-\gamma }\left(
r^{\left( d-1\right) /\left( d+1\right) }t\right) ^{\left( d-1\right)
/2}\right) \\
&\leq C\left\vert \Omega \left( r,t\right) \right\vert \left( r^{\left(
d-1\right) /\left( d+1\right) }t\right) ^{\left( d-1\right) /\left( 2\gamma
+2\right) } \\
&\leq C\left\vert \Omega \left( r,t\right) \right\vert \left( t^{1-\left(
d-1\right) /\left( \left( d+1\right) \alpha \right) }\right) ^{\left(
d-1\right) /\left( 2\gamma +2\right) }.
\end{align*}

Finally, in order to prove the lemma it suffices to choose 
\begin{equation*}
\beta \leq \frac{\left( \alpha -\dfrac{d-1}{d+1}\right) \left( d-1\right) }{%
\alpha \left( 2\gamma +2\right) }.
\end{equation*}
\end{proof}

\begin{proof}[Proof of Theorem \ref{T1}] By the previous lemmas, choosing $ \beta<1$ in Lemma \ref{L8}, one has  
\begin{align*}
|W(r,t)|+|Z(r,t)|+|Y(r,t)|&\leq C|\Omega(r,t)|\left(t\log(2+1/t)+tr^{-1}\log(2+1/t)+t^\beta\right)\\
&\leq C|\Omega(r,t)|t^\beta
\end{align*}
\end{proof}

We conclude with some remarks.

\begin{remark}\label{R1}
As said in the introduction, for the validity of the
theorem the assumption that the widths of the annuli converge to zero does
not suffice, and one has to require a suitable speed. Indeed in \cite{PS}
a somehow stronger failure of an asymptotic estimate is
proved. In any dimension $d$ the variance of spherical annuli $\left\{
r-t/2 < \left\vert x\right\vert \leq r+t/2\right\} $ is always smaller
than $Cr^{d-1}t$, and for some sequences $r\rightarrow +\infty $ it is
larger than $cr^{d-1}t$. Moreover, there exist sequences $r\rightarrow
+\infty $ and $t\rightarrow +\infty $ with associated variance much smaller
than $cr^{d-1}t$ for every $c>0$. In dimension $d\equiv 3$ modulo 4 this
also holds for some sequences of widths that stay bounded or that tend to
zero slower than any negative power of the radii. This is related to the
location of the zeroes of the Fourier transform of an annulus. See also \cite{BCGT}
for related results on higher order
moments.
\end{remark}

\begin{remark}\label{R2} 
As said in the introduction, the variance of annuli with
boundary points of zero curvature may be much larger than the mean, and an
asymptotic estimate of the variance may fail. A simple example are the flat
annuli in the plane generated by squares with sides parallel to the axes, 
\begin{align*}
A&=\left\{ x=\left( x_{1},x_{2}\right) :\ n-t/2<\max \left\{ \left\vert
x_{1}\right\vert ,\left\vert x_{2}\right\vert \right\} \leq n+t/2\right\} ,\\
B&=\left\{ x=\left( x_{1},x_{2}\right) :\ n<\max \left\{ \left\vert
x_{1}\right\vert ,\left\vert x_{2}\right\vert \right\} \leq n+t\right\} .
\end{align*}

The diameters and thicknesses of these two annuli are approximately the
same, but the random variables that count the lattice points are quite
different when $n$ is a large integer and $t$ is a small positive number.
The random variable $N(A,x)$ that counts the number of integer points in $%
A-x $ takes the value $8n$ on a set with measure $t^{2}$, the value $4n$ on
a set with measure $2t-2t^{2}$, and $0$ otherwise, and the mean and variance
are 
\begin{align*}
\int_{\mathbb{T}^{2}}N(A,x)dx&=8nt,\\
\int_{\mathbb{T}^{2}}\left\vert N(A,x)-8nt\right\vert
^{2}dx&=32n^{2}t-32n^{2}t^{2}\sim 32n^{2}t.
\end{align*}

Similarly, the random variable $N(B,x)$ that counts the number of integer
points in $B-x$ takes the value $4n+1$ on a set with measure $4t^{2}$, the
value $2n$ on a set with measure $4t-8t^{2}$, and $0$ otherwise, and the
mean and variance are 
\begin{align*}
\int_{\mathbb{T}^{2}}N(B,x)dx&=8nt+4t^{2}\sim 8nt,\\
\int_{\mathbb{T}^{2}}\left\vert N(B,x)-\left( 8nt+4t^{2}\right) \right\vert
^{2}dx&=16n^{2}t-32n^{2}t^{2}+32nt^{2}-64nt^{3}+4t^{2}-16t^{4}\\
&\sim
16n^{2}t.
\end{align*}

Observe that the means of $N(A,x)$ and $N(B,x)$ are approximately the same
and they are much smaller than the variances, and that the variance of $%
N(A,x)$ is about twice the variance of $N(B,x)$. In particular, the
variances of these flat annuli have a sort of oscillating behavior.
\end{remark}

\begin{remark}\label{R3} 
The above are estimates of the discrepancy between volume
and integer points in translated annuli. As in \cite{HR, M, Mi, Si}, one may ask about similar estimates
when the annuli are not translated and the averages are with respect to
dilations. We suspect that the discrepancy with respect to dilations of spherical annuli may be much larger than the discrepancy with
respect to translations, and indeed in \cite{HR} it is proved that this is the case for annuli in the plane, that is in dimension $d=2$. 
\end{remark}

\bibliographystyle{amsplain}

\begin{thebibliography}{99}

\bibitem{BF}
 {T. Bonnesen},  {W. Fenchel}, {``Theory of convex bodies"},
BCS Associates, Moscow, ID, 1987.

\bibitem{BCGT}
 {L. Brandolini},  {L. Colzani},  {G. Gigante},  {%
G. Travaglini}, $L^{p}$\textit{\ and }$Weak-L^{p}$\textit{\
estimates for the number of integer points in translated domains},
Math. Proc. Cambridge Philos. Soc. {\bf 159} (2015),
471--480.

\bibitem{CLM}
 {Z. Cheng},  {J. L. Lebowitz},  {P. Major}, \textit{On the
number of lattice points between two enlarged and randomly shifted copies of
an ova}, Probab. Theory Related Fields {\bf 100} (1994), 253--268.

\bibitem{H1}
 {C. S. Herz}, \textit{Fourier Transforms Related to Convex Sets},
Ann. of Math. {\bf 75} (1962), 81--92.

\bibitem{H2}
 {C. S. Herz}, \textit{On the number of lattice points in a convex set}%
, Amer. J. Math. {\bf 84} (1962), 126--133.

\bibitem{Hl}
 {E. Hlawka}, \textit{Uber Integrale auf convexen Korpen, I, II},
Monatsh.  Math. {\bf 54} (1950), 1--36, 81--99.

\bibitem{Ho}
 {L. H\"{o}rmander}, {``The analysis of linear partial
differential operators I - Distribution theory and Fourier analysis"},
Springer-Verlag, Berlin, 1983.

\bibitem{HR}
 {C. P. Hughes},  {Z. Rudnik}, \textit{On the distribution of
lattice points in thin annuli}, Int. Math. Res. Not. IMRN
{\bf 13} (2004), 637--657.

\bibitem{K}
 {D. Kendall}, \textit{On the number of lattice points inside a random
oval}, Q. J. Math. {\bf 19} (1948), 1--26.

\bibitem{Ki}
 {C. O. Kiselman}, \textit{Smoothness of vector sums of plane convex
sets}, Math. Scand. {\bf 60} (1987), 239--252.

\bibitem{KP}
 {S. G. Krantz},  {H. R. Parks}, \textit{On the vector sum of two
convex sets in space}, Canadian J. Math. {\bf 43} (1991), 347--355.

\bibitem{Li}
 {W. Littman}, \textit{Fourier transforms of surface-carried measures
and differentiability of surface averages}, Bull. Amer.
Math. Soc. {\bf 69} (1963), 766--770.

\bibitem{M}
 {P. Major}, \textit{Poisson law for the number of lattice points in a random strip with finite area},
Probab. Theory Related Fields {\bf 92} (1992),  423--€"464.

\bibitem{Mi}
 {N.Minami},
\textit{On the Poisson limit teorems of Sinai and Major},
Commun. Math. Phys. {\bf 213} (2000),  203--€"247.

\bibitem{PS}
 {L. Parnovski},  {N. Sidorova}, \textit{Critical dimensions for
counting lattice points in Euclidean annuli}, Math. Model. 
Nat. Phenom. {\bf 5} (2010), 293--316.

\bibitem{Sc}
 {R. Schneider}, {``Convex bodies: the Brunn Minkowski theory"},
Cambridge University Press, Cambridge, 2014.

\bibitem{Si}
 {Y. G. Sinai}, \textit{Poisson distribution in a geometric problem},
Adv. Soviet Math. {\bf 3} (1991), 199--214.

\bibitem{SW}
 {E. M. Stein},  {G. Weiss}, {``Introduction to Fourier
analysis on Euclidean spaces"}, Princeton University Press, Princeton, NJ, 1971.

\bibitem{St}
 {E. M. Stein}, {``Harmonic analysis, real variable methods,
orthogonality, and oscillatory integrals"}, Princeton University Press, Princeton, NJ
1993.

\end{thebibliography}

\end{document}